\documentclass[11pt,a4paper]{article}

\usepackage{amsfonts,amsmath,comment,verbatim,amsthm,amssymb,graphicx,fancyhdr,
makeidx,amscd}

\usepackage[all]{xy}
\numberwithin{equation}{section}
\newtheorem{definition}{Definition}
\newtheorem{proposition}{Proposition}
\newtheorem{lemma}{Lemma}
\newtheorem{theorem}{Theorem}
\newtheorem{remark}{Remark}
\newtheorem{corollary}{Corollary}
\newtheorem{example}{Example}

\newtheorem*{conjecture}{Conjecture}
\newtheorem*{question}{Question}
\numberwithin{equation}{section}
\begin{document}

\newcommand{\riem}{(M^m, \langle \, , \, \rangle)}
\newcommand{\Hess}{\mathrm{Hess}\, }
\newcommand{\hess}{\mathrm{hess}\, }
\newcommand{\cut}{\mathrm{cut}}
\newcommand{\ind}{\mathrm{ind}}
\newcommand{\ess}{\mathrm{ess}}
\newcommand{\longra}{\longrightarrow}
\newcommand{\eps}{\varepsilon}
\newcommand{\ra}{\rightarrow}
\newcommand{\hra}{\hookrightarrow}
\newcommand{\vol}{\mathrm{vol}}
\newcommand{\di}{\mathrm{d}}
\newcommand{\R}{\mathbb R}
\newcommand{\C}{\mathbb C}
\newcommand{\Z}{\mathbb Z}
\newcommand{\N}{\mathbb N}
\newcommand{\HH}{\mathbb H}
\newcommand{\grad}{{{\rm{grad}}}}
\newcommand{\esse}{\mathbb S}
\newcommand{\bull}{\rule{2.5mm}{2.5mm}\vskip 0.5 truecm}
\newcommand{\binomio}[2]{\genfrac{}{}{0pt}{}{#1}{#2}} 
\newcommand{\metric}{\langle \, , \, \rangle}
\newcommand{\lip}{\mathrm{Lip}}
\newcommand{\loc}{\mathrm{loc}}
\newcommand{\diver}{\mathrm{div}}
\newcommand{\disp}{\displaystyle}
\newcommand{\rad}{\mathrm{rad}}
\newcommand{\Ricc}{\mathrm{Ricc}}
\newcommand{\mmetric}{\langle\langle \, , \, \rangle\rangle}
\newcommand{\sn}{\mathrm{sn}}
\newcommand{\cn}{\mathrm{cn}}
\newcommand{\ink}{\mathrm{in}}
\newcommand{\hol}{\mathrm{H\ddot{o}l}}
\newcommand{\capac}{\mathrm{cap}}
\newcommand{\bmo}{\{b <0\}}
\newcommand{\bmuo}{\{b \le 0\}}

\newcommand{\essem}{\mathds{S}^m}
\newcommand{\erre}{\mathds{R}}
\newcommand{\errem}{\mathds{R}^m}
\newcommand{\enne}{\mathds{N}}
\newcommand{\acca}{\mathds{H}}

\newcommand{\cvett}{\Gamma(TM)}
\newcommand{\cinf}{C^{\infty}(M)}
\newcommand{\sptg}[1]{T_{#1}M}
\newcommand{\partder}[1]{\frac{\partial}{\partial {#1}}}
\newcommand{\partderf}[2]{\frac{\partial {#1}}{\partial {#2}}}
\newcommand{\ctloc}{(\mathcal{U}, \varphi)}
\newcommand{\fcoord}{x^1, \ldots, x^n}
\newcommand{\ddk}[2]{\delta_{#2}^{#1}}
\newcommand{\christ}{\Gamma_{ij}^k}
\newcommand{\ricc}{\operatorname{Ricc}}
\newcommand{\supp}{\operatorname{supp}}
\newcommand{\sgn}{\operatorname{sgn}}
\newcommand{\rg}{\operatorname{rg}}
\newcommand{\inv}[1]{{#1}^{-1}}
\newcommand{\id}{\operatorname{id}}
\newcommand{\jacobi}[3]{\sq{\sq{#1,#2},#3}+\sq{\sq{#2,#3},#1}+\sq{\sq{#3,#1},#2}
=0}
\newcommand{\lie}{\mathfrak{g}}
\newcommand{\Lim}{{\rm Lim }\; \varphi }
\newcommand{\wedgedot}{\wedge\cdots\wedge}
\newcommand{\rp}{\erre\mathds{P}}
\newcommand{\II}{\operatorname{II}}
\newcommand{\gradh}[1]{\nabla_{H^m}{#1}}
\newcommand{\absh}[1]{{\left|#1\right|_{H^m}}}
\newcommand{\mob}{\mathrm{M\ddot{o}b}}
\newcommand{\mab}{\mathfrak{m\ddot{o}b}}
\newcommand{\foc}{\mathrm{foc}}
\newcommand{\F}{\mathcal{F}}
\newcommand{\Cf}{\mathcal{C}_f}
\newcommand{\cutf}{\mathrm{cut}_{f}}
\newcommand{\Cn}{\mathcal{C}_n}
\newcommand{\cutn}{\mathrm{cut}_{n}}
\newcommand{\Ca}{\mathcal{C}_a}
\newcommand{\cuta}{\mathrm{cut}_{a}}
\newcommand{\cutc}{\mathrm{cut}_c}
\newcommand{\cutcf}{\mathrm{cut}_{cf}}
\newcommand{\rk}{\mathrm{rk}}
\newcommand{\crit}{\mathrm{crit}}
\newcommand{\diam}{\mathrm{diam}}
\newcommand{\haus}{\mathcal{H}}
\newcommand{\po}{\mathrm{po}}
\newcommand{\gr}{\mathcal{G}}
\newcommand{\metrict}{( \, , \, )}
\newcommand{\Pro}{\mathbb{P}}
\newcommand{\Bo}{\mathbb{B}}

\author{G. Pacelli Bessa \and Luqu\'{e}sio P. Jorge \and Luciano
Mari}
\title{\textbf{On the spectrum of  bounded immersions}}
\date{}
\maketitle
\scriptsize \begin{center} Departamento de Matem\'atica,
Universidade Federal do Cear\'a\\
Campus do  PICI, 60455-760 Fortaleza-Ce, (Brazil)\\
E-mail address: bessa@mat.ufc.br, ljorge@mat.ufc.br, lucio.mari@libero.it
\end{center}

\normalsize

\vspace{0.5cm}

\begin{abstract} In this paper, we investigate the relationship between the
discreteness of the spectrum of a non-compact, extrinsically bounded submanifold
$\varphi \colon M^m \ra N^n$ and the Hausdorff dimension of its limit set
$\lim\varphi$. \hspace{-1mm}In particular, we prove that if $\varphi \colon \!M^2 \ra D \subseteq
\R^3$ is a minimal immersion into an open, bounded, strictly convex subset $D$
with $C^2$-boundary, then $M$ has discrete spectrum provided that
$\haus_\Psi(\lim\varphi \cap D)=0$, where $\haus_\Psi$ is the generalized
Hausdorff measure of order $\Psi(t) = t^2|\log t|$. Our theorem applies to a
number of examples recently constructed by various authors in the light of N.
Nadirashvili's discovery of complete, bounded minimal disks in $\R^3$, as well
as to solutions of Plateau's problems, giving a fairly complete
answer to a question posed by S.T. Yau in his Millenium Lectures. Suitable counter-examples show the sharpness of our results: in particular, we develop a simple criterion for the existence of essential spectrum which is suited for the techniques
developed after Jorge-Xavier and Nadirashvili's examples.
\end{abstract}

\section{Introduction}\hspace{.5cm}
The Calabi-Yau conjectures have their origin in a set of two problems proposed
by E. Calabi in the 1960's, about the non-existence of complete minimal
hypersurfaces of $\mathbb{R}^{n}$ subjected to certain extrinsic bounds
(\cite{calabi2}, see also \cite[p. 212]{chern}). \begin{itemize}\item Calabi
proposed the first conjecture as an exercise. He wrote:  \textquotedblleft {\em Prove that
any  complete minimal hypersurface in $\mathbb{R}^{n}$ must be
unbounded.}\textquotedblright
\item The second problem, on the other hand, was  proposed almost as an unlikely
conjecture. He wrote
 \textquotedblleft {\em  A more ambitious conjecture is: a complete  minimal
hypersurface in $\mathbb{R}^{n}$ has an unbounded projection in every $(n -
2)$-dimensional flat subspace.}\textquotedblright
\end{itemize} It is known by works of  L. Jorge-F. Xavier  \cite{jorge-xavier-annals} and N.
Nadirashvili   \cite{Nadirashvili} that both conjectures turned out to be false. More precisely, Jorge-Xavier constructed a
non-flat, complete minimal surface lying between two parallel planes in
$\mathbb{R}^{3}$, showing that the second conjecture was false in general,
whereas  N. Nadirashvili  constructed a bounded, complete   minimal immersion of
the unit disk $\mathbb{D}$  into $\mathbb{R}^{3}$, contradicting the statement
of the first conjecture. Both  counter-examples lie on the construction of
suitable labyrinths of compact subsets and clever use of Runge's Theorem in
order to find appropriate Weierstrass data (the interested reader could  consult
\cite{collinrosenberg} for a detailed proof of the main theorems in
\cite{Nadirashvili}). In view of these results, in his Millennium Lectures
\cite{Yau3}, \cite{yau4}, S. T. Yau  \textquotedblleft
calibrated\textquotedblright \hspace{.1mm}
the Calabi conjectures, proposing a new set of questions about bounded minimal
hypersurfaces.\begin{itemize}\item[]  He wrote: \textquotedblleft {\em  It is
known \cite{Nadirashvili} that there are complete minimal surfaces properly
immersed into
the $[$open$]$ ball.  What is the geometry of these surfaces?  Can they be
embedded?   Since the
curvature must tend to minus infinity, it is important to find the precise
asymptotic
behavior of these surfaces near their ends. Are their $[$Laplacian$]$ spectra
discrete?\textquotedblright}. \end{itemize}
 These questions became known in the literature  as the Calabi-Yau conjectures
on minimal hypersurfaces. The problem about the existence of a bounded,
complete,  \emph{embedded}  minimal surfaces in $\mathbb{R}^{3}$ was negatively
answered by T. Colding-W. Minicozzi  in  the finite topology case, see
\cite{colding-minicozzi}. Although Yau's question suggests that Nadirashvili's
example is \emph{properly} immersed into an open ball $D \subset \mathbb{R}^{3}$
(that is, the pre-image of compact subsets  of $D$ are compacts in the surface),
this is not  clear and one may consider,  as the first  problem in the
Calabi-Yau conjectures, the question:  \textquotedblleft {\em Does there  exist
a complete minimal surface properly immersed into a ball of
$\mathbb{R}^{3}$?\,\textquotedblright .}

This question, and more generally the search for understanding the shape of the
limit set of bounded minimal surfaces, has stimulated  intense research in the
last fifteen years\footnote{Given a map $\varphi : M \ra N$ between topological
manifolds, the limit set of $\varphi$, $\lim \varphi$, is defined as
\begin{eqnarray}
\lim \varphi = \{x\in  N ; \;\exists \,\{p_{n}\}\subseteq M \text{ divergent in
}M, \text{ such that } \varphi (p_{n}) \ra x \text{ in }N\big\}.
\end{eqnarray}
By this definition, a map $\varphi : M \ra N$ whose image lie in a bounded, open
subset $D\subseteq N$ is proper in $D$ if and only if $\lim \varphi \subseteq
\partial D$.}. Regarding properness, we briefly recall the main achievements:
\begin{itemize}
\item[$(i)$] with a highly nontrivial refinement of Nadirashvili's technique, F.
Martin and S. Morales in \cite{pacomartin} proved that, for each convex domain
$D \subseteq \mathbb{R}^{3}$ not necessarily bounded or smooth, there exists a
complete  minimal disk $B \subseteq \C$ properly immersed into $D$. Later on, M.
Tokuomaru in \cite{tokuomaru} constructed a complete minimal annulus which is
proper in the unit ball of $\R^3$;
\item[$(ii)$] in \cite{pacomartin2}, Martin-Morales improved the results in
\cite{pacomartin}  showing that, if $D$ is a bounded strictly convex domain of
$\mathbb{R}^{3}$ with $\partial D$ is of class $C^{2,\alpha}$, then there exists
a complete,  minimal disk properly immersed into $D$ whose limit set is close to
a prescribed Jordan curve\footnote{We recall that a Jordan curve is defined as a
continuous embedding of $I \subseteq \R$ (or $I\subseteq \esse^1$) in $\R^3$.}
on $\partial D$;
\item[$(iii)$] Alarcon, Ferrer and  Martin in \cite{alarcon-Ferrer-Martin-GAFA}
extended the results in \cite{pacomartin} and \cite{tokuomaru} from disks and
annuli to open surfaces $M$  with any finite topology;
\item[$(iv)$] improving on \cite{pacomartin2}, Ferrer, Martin and Meeks in
\cite{ferrer-martin-meeks} showed that if $D \subset \R^3$ is bounded, convex
and smooth, then every open surface $M$ with finite topology can be properly and
minimally immersed into $D$ in such a way that the limit sets of the ends are
disjoint compact connected subsets of $\partial D$. Furthermore (see Proposition
1 in \cite{ferrer-martin-meeks}), they proved that, for every convex, open set
$D$ and every non-compact, orientable surface $M$, there exists a complete,
proper minimal immersion $\varphi : M \ra D$ such that $\lim \varphi \equiv
\partial D$.
\end{itemize}

A parallel line of research focused on the problem of controlling the size of
$\lim \varphi$ from the measure-theoretic point of view. In this respect,
\begin{itemize}
\item[$(v)$] Martin and Nadirashvili, \cite{martin-nadirashvili}, constructed
complete minimal immersions $\varphi\colon B \rightarrow \R^{3}$ of the unit
disk $B \subseteq \C$ admitting continuous extensions to the closed disk
$\overline{\varphi}\colon \overline{B}\rightarrow \mathbb{R}^{3}$ in such a way
that $\overline{\varphi}_{\vert \partial B}\colon \partial B = \mathbb{S}^1\to
\overline{\varphi}(\mathbb{S}^1)$ is an homeomorphism and
$\overline{\varphi}(\mathbb{S}^{1})$ is a non-rectifiable Jordan curve of
Hausdorff dimension $\dim_\haus(\overline{\varphi}(\mathbb{S}^{1})) = 1$.
Moreover, they showed that the set of Jordan curves
$\overline{\varphi}(\mathbb{S}^{1})$ constructed via this procedure
is dense in the space of Jordan curves of $\mathbb{R}^{3}$ with respect to the
Hausdorff metric.
\end{itemize}
In this paper, we address Yau's question of deciding whether the spectrum of
bounded, minimal surfaces is discrete or not, and we provide a sharp, general
criterion that applies to each of the examples in $(i), \ldots, (v)$, as well as
to many other minimal surfaces. The first answer to this question was given in
Bessa-Jorge-Montenegro in \cite{bessa-Jorge-montenegro-JGA}, where they proved
that the spectrum of a complete minimal surface properly immersed into a ball of
$\mathbb{R}^{3}$ is always discrete. Despite the generality of this result,
there is the technical unpleasant drawback that the possible convex ambient
domains $D \subset \R^3$ are restricted to balls. Much more important, their
approach uses in a crucial way the properness condition, and cannot be
generalized to deal with non-proper immersions. On the other hand, it is
reasonable to guess that a well-behaved limit set, close (loosely speaking) to
a smooth curve as in $(v)$, could imply that $\Delta$ has discrete spectrum on
$M$. In some
sense, we could say that the minimal surface is not too far from a compact
surface with boundary. Indeed, it is not necessary that $\lim\varphi$ resembles
a curve from the measure-theoretic point of view, but a much more general
condition is sufficient:

\begin{theorem}\label{teo_maindim2}

Let $\varphi \colon M^2 \to D \subset N$ be a minimal immersion into an open,
bounded, $2$-convex subset of a Cartan-Hadamard manifold $N$. Set $\Psi(t) =
t^2|\log t|$. If the $\Psi$-Hausdorff measure of $\lim\varphi \cap D$ satisfies
$\haus_\Psi(\lim\varphi \cap D) = 0$, then the spectrum of $-\Delta$ is discrete.
\end{theorem}

\begin{remark}\label{rem_spectrum}
\emph{We stress that no geodesic completeness of $M$ is required in Theorem
\ref{teo_maindim2}. When $M$ is incomplete, $-\Delta$ may fail to be essentially
self-adjoint on $C^\infty_c(M)$, in sharp contrast with the complete case where the self-adjointness of $-\Delta$ is automatic. For this reason, we agree on always considering the Friedrichs extension of $\Delta$, that is, the
unique self-adjoint extension of $(-\Delta, C^\infty_c(M))$ whose domain lies in
that of the closure of the quadratic form $Q(u,v) = (-\Delta u,v)_{L^2}$, for
$u,v \in C^\infty_c(M)$.
}
\end{remark}

\begin{remark}
\emph{We also underline that there is no codimension restriction, that is, $N$
may be of any dimension.}
\end{remark}

Later, we will recall the notion of $2$-convexity for subsets with $C^2$
boundary, and the definition of Hausdorff measure $\haus_\Psi$. Here, we just
observe that $2$-convexity relaxes standard convexity, so $D$ can be
any bounded strictly convex subset of $M$ with $C^2$-boundary. As for
$\haus_\Psi$, condition $\haus_\Psi(\lim\varphi \cap D) = 0$ is satisfied, for
instance, whenever the Hausdorff dimension of $\lim\varphi \cap D$ is strictly
less than $2$. It is worth to underline that we are considering only the part of
$\lim\varphi$ \emph{inside} $D$: there is no requirement on the portion of
$\lim\varphi$ contained in $\partial D$. As an immediate corollary, we have the
following

\begin{corollary}

All the examples of complete, bounded minimal surfaces constructed in $(i),
\ldots, (v)$ $($if the convex sets are bounded and $C^2)$ have discrete
spectrum.

\end{corollary}

The fact that the completeness of $M$ is irrelevant enables us to apply our
result to solutions of the Plateau problem for arbitrary Jordan curves, whose
existence is granted by R. Douglas in \cite{douglas} (see also the treatises
\cite{courant} and \cite{dierkeseco, dierkeseco2}, as well as \cite{jost},
\cite{tomitromba}):

\begin{corollary}\label{cor_plateau}

Let $\Gamma \colon \esse^1 \ra \R^3$ be a Jordan curve. If
$\haus_\Psi(\Gamma(\esse^1)) = 0$, where $\Psi(t) = t^2 |\log t|$, then every
minimal surface spanning $\Gamma$ has discrete spectrum.
\end{corollary}

\begin{remark}
\emph{It is known (see \cite{osgood}) that there exist Jordan curves whose image have non-zero $2$-dimensional Hausdorff measure, whence the measure condition in Corollary \ref{cor_plateau} is not automatically satisfied.
}
\end{remark}

There is a trivial reason why the estimate in Theorem \ref{teo_maindim2} is
sharp, and it is the presence of infinite-sheet coverings.

\begin{example}

\emph{Consider a bounded, complete minimal annulus $f \colon A \ra \R^3$ (either
the proper one constructed in \cite{tokuomaru}, or the one in
\cite{martinmorales}). Then, taking the universal covering $\pi\colon M \ra A$,
$\varphi = f \circ \pi$ is a bounded, complete minimal surface with non-empty
essential spectrum, as it is for every infinite-sheet covering (for instance,
one can note that $M$ has the ball property, see Definition \ref{def_ball} and
Corollary \ref{cor_disk} below). Clearly, $\dim_\haus (\lim \varphi) \ge 2$
since, by construction, $\lim \varphi \supseteq f(A)$.
}
\end{example}

As we shall see, Theorem \ref{teo_maindim2} is a particular case of a more
general result, Theorem \ref{teo_maingenerale} below, where we deal with
complete submanifolds $\varphi \colon M^m \to  N^n$ immersed with  sufficiently
small mean curvature in an $m$-convex and regular ball $B_R \Subset N$. Thus,
the property described in Theorem \ref{teo_maindim2} does not, indeed,
exclusively pertain to the realm of minimal surfaces. For technical reason, we
postpone the general statement of our main theorem to Section
\ref{sec_statement}.
\begin{remark}
\emph{As a consequence of Theorem \ref{teo_main}, an analogous of
Corollary \ref{cor_plateau} holds when the mean curvature is non-zero and
sufficiently small. In this case, the existence of disks solving Plateau's
problem is granted by H. Werner in \cite{werner}. We leave the correspondent
statement to the interested reader.}
\end{remark}
To inspect more closely the sharpness of our results, Section \ref{sec_sharpness} is devoted to find sufficient conditions for a manifold to have non-empty essential spectrum. These conditions will be used to investigate minimal submanifolds subjected to some general extrinsic constrain, such as, for instance, Jorge-Xavier minimal surface in a slab. In particular, we shall concentrate on criteria that do not involve the curvature of the submanifold, as this is often a hardly available information: for example, it seems extremely difficult to control the curvature when exploiting the techniques used for the examples in $(i),\ldots , (v)$. We are led to the following

\begin{definition}\label{def_ball}
A Riemannian manifold $(M, \metric)$ is said to have the \emph{\textbf{ball property}} if there exists $R>0$ and a
collection of disjoint balls $\{B_R(x_j)\}_{j=1}^{+\infty}$ of radius $R$ for
which, for some constants $C>0$, $\delta \in (0, 1)$ possibly depending on $R$,
\begin{equation}\label{desig1}
\vol \big(B_{\delta R}(x_j)\big) \ge C^{-1} \vol\big(B_R(x_j)\big) \qquad
\forall \, j \in \N
\end{equation}
\end{definition}

\begin{remark}
 \emph{Note that \eqref{desig1} is \emph{not} a doubling condition, as $C$ may depend on $R$.
}
\end{remark}

\begin{proposition}\label{lem_crite}
If a Riemannian manifold has the ball property (with parameters $R, \delta, C$), then
\begin{equation}\label{ifsigmaess}
\inf \sigma_\ess(-\Delta) \le \frac{C}{R^2(1-\delta)^2}.
\end{equation}
\end{proposition}

The ball property is reasonably easy to check, since it only requires a $C^0$ control on the metric in some region of $M$, and is very suited to deal with surfaces constructed via developments of Jorge-Xavier original technique in \cite{jorge-xavier-annals}. For example, we easily prove the next

\begin{proposition}\label{prop_jx}
For a suitable choice of the parameters, the Jorge-Xavier complete minimal
surface $(M^2, \di s^2)$ in a slab of $\R^3$, constructed in
\cite{jorge-xavier-annals}, satisfies $\inf \sigma_\ess(-\Delta)=0$.
\end{proposition}

For more results and comments, we refer to Section \ref{sec_sharpness}. There, the reader can also find some open problems that we believe to be worth investigating.

\section{Setting, notations and preliminaries}
Let $M$, $N$ be connected (smooth) Riemannian manifolds of dimensions $m$ and
$n$ respectively and let $\varphi\colon M \rightarrow N$ be an isometric
immersion. In what follows, $\nabla \di \varphi$ will stand for the second
fundamental form of $\varphi$ and $H = \frac{1}{m} \mathrm{tr}(\nabla \di
\varphi)$ for its mean curvature. All the elements describing the Riemannian
structure of $N$ will be marked with a bar, so that, for instance, $\overline
\nabla$ and $\overline K$ will denote the Riemannian connection and the
sectional curvature of $N$. Having fixed $o \in N$, we will write $\overline
K_\rad(x)$ for the radial sectional curvature of $N$, that is, the restriction
of $\overline K(x)$ to the subset of $2$-planes at $x$ containing tangent
vectors of minimizing geodesics issuing from $o$. For $x,y \in N$, we define
$\rho(x,y)= \mathrm{dist}(x,y)$, and $\rho(x)= \mathrm{dist}(x,o)$ whenever the
second point is considered as fixed. The symbol $B_r(x)$ indicates the ball of
radius $r$ centered at $x$,
and we simply write $B_r$ when $x=o$. Similarly, for $A\subseteq N$ the symbol
$B_r(A)$ denotes the open set of points whose distance from $A$ is less than
$r$. We use $\R^+$ and $\R^+_0$ to denote, respectively, $(0,+\infty)$ and
$[0,+\infty)$.

Hereafter, we will consider a relaxed notion of convexity, called $j$-convexity,
which is widely used in the literature (for instance, see \cite{jorgetomi}).
\begin{definition}\label{def_convexity}
For an integer $j \le n$, We say that an open subset $D \subseteq N^n$ is
(strictly) $j$-convex if there exists $F \in C^2(N)$, satisfying the following
properties:
\begin{itemize}
\item[$(i)$] $F<0$ on $D$, $F=0$ on $\partial D$;
\item[$(ii)$] for every $p\in D$, denoting with $\lambda_1(p) \le \lambda_2(p)
\le \ldots \le \lambda_{n}(p)$ the eigenvalues of the Hessian $\nabla \di F$ at
$p$ written in increasing order, it holds $\lambda_1(p) +\ldots + \lambda_j(p)
>0$.
\end{itemize}
The set $D$ will be called uniformly convex with constant $c>0$ if at every
point $p \in D$ it holds
$\lambda_1(p)+\ldots+\lambda_j(p) \ge c$. For a uniformly $j$-convex subset $D$,
the pair $(F,c)$ will be called the data of uniform convexity of $D$.
\end{definition}

In particular, if $j=1$, the definition of $j$-convexity coincides with the
standard definition of (strict) convexity, and for this reason we simply say
that $D$ is convex. Moreover, we observe that $n$-convexity means that $D$ is
$F^{-1}((-\infty,0))$ for a strictly subharmonic function $F$. Since each
$\lambda_i$ is a Lipschitz function of $p \in N$, by a compactness argument the
properties of $j$-convexity and uniform $j$-convexity coincide for relatively
compact $D$.

\begin{example}\label{ex_bolle}
\emph{
Let $N$ be a complete Riemannian manifold, let $o \in N$ be a reference origin
and suppose that $\overline K_\rad(x) \le -G(\rho(x))$, for some $G \in
C^0(\R^+_0)$. Consider a solution $h(t)\in C^2(\R^+_0)$ of
\begin{equation}\label{defh}
\left\{\begin{array}{l}
   h''(t)-G(t)h(t)= 0, \\[0.1cm]
   h(0)=0, \quad h'(0)=1,
  \end{array}\right.
\end{equation}
and let $R$ be such that $h'>0$ on $[0,R]$. Suppose that $B_R \subseteq N$ does
not intersect the cut-locus $\cut(o)$, that is, that $B_R$ is a regular ball of
$N$. Then, by the Hessian comparison theorem (see for instance \cite{prs},
Theorem 2.3), for $x \in B_R$
$$
\nabla \di \rho \ge \frac{h'(\rho)}{h(\rho)} \Big( \metrict - \di \rho \otimes
\di \rho\Big), \qquad \rho = \rho(x).
$$
Having set
$$
f(t) = \int_0^t h(s) \di s, \qquad F(x) = f(\rho(x)) \quad \text{on } B_R,
$$
then by the chain rule
$$
\nabla \di F = f''(\rho) \di \rho \otimes \di \rho + f'(\rho) \nabla \di \rho
\ge h'(\rho) \metrict \ge \left[ \inf_{[0,R]}h'\right] \metrict,
$$
whence the ball $B_R \subseteq N$ is convex, with constant $c=\inf_{[0,R]}h'$.
For instance, if $G(t)=k\le 0$ is constant, then every ball centered at a pole
$o$ of $N$ is convex (we recall that a pole is a point $o \in N$ such that
$\cut(o)=\emptyset$, or equivalently such that $\exp_o: T_oN \ra N$ is a
diffeomorphism). The same happens when $k > 0$ whenever $B_R \subseteq N$ is a
regular ball of radius $R < \pi/2\sqrt{k}$. If $o$ is a pole for $N$, there are
some sufficient conditions ensuring that each $B_R$ is convex. For a general
case, we suggest the reader to consult Remark \ref{rem_ipoGh} below.
}
\end{example}

Before stating our result in its stronger form, we also recall some general
notions on Hausdorff measures. We follow the exposition in \cite{mattila},
Chapter 4, although with a different notation. According to Carath\'eodory
construction, we consider a function $\Psi \ge 0$ defined and continuous on some
right neighborhood $[0,2\delta_0)$ of zero, and such that $\Psi(0)=0$, together
with a family $\F$ of Borel subsets of $M$ satisfying the following property:
\begin{itemize}
\item[] For every $\delta \in (0, \delta_0)$, there exist $\{E_i\} \subseteq \F$
such that $M = \bigcup_{i=1}^{+\infty} E_i$ and $\diam(E_i) \le \delta$.
\end{itemize}
For each $\delta \in (0, \delta_0)$ and for every subset $A \subseteq M$, we set
\begin{equation}\label{defhaus}
\begin{array}{l}
\disp \haus_{\Psi, \delta}(A) = \disp \inf\left\{ \sum_{i=1}^{+\infty} \Psi
\big(\diam(E_i)\big) \ : \ \{E_i\} \subseteq \F, \  A \subseteq
\bigcup_{i=1}^{+\infty} E_i, \ \diam(E_i) \le \delta \right\} \\[0.5cm]
\disp \haus_\Psi(A) = \lim_{\delta \downarrow 0}\haus_{\Psi, \delta}(A) =
\sup_{\delta \in (0, \delta_0)} \haus_{\Psi, \delta}(A).
\end{array}
\end{equation}
If $\F \equiv \{\text{borel subsets of } M\}$,  then the resulting measure
$\haus_\Psi$ is a borel regular measure (\cite{mattila}, Theorem 4.2), and we
call it the Hausdorff measure related to $\Psi$. By Theorem 4.4 in
\cite{mattila}, the same $\haus_\Psi$ can be obtained if we restrict to the
subfamily $\F = \{\text{open subsets of } M\}$. The particular choice $\Psi(t) =
t^\beta$, for some fixed $\beta>0$, gives the standard Hausdorff measure
$\haus^\beta$ of order $\beta$, up to an unessential constant.

\begin{remark}\label{rem_haus}
\emph{If we restrict $\F$ to the subfamily of geodesic balls of $M$, the
resulting measure $\overline\haus_\Psi$ does not coincide, in general, with
$\haus_\Psi$ (see \cite{mattila}, Chapter 5). However, if for some constant
$C>0$ it holds
\begin{equation}\label{propper}
\Psi(2t) \le C\Psi(t)  \qquad \text{for } t \in (0, \delta_0),
\end{equation}
then $\haus_\Psi \le \overline \haus_\Psi \le C\haus_\Psi$. The first inequality
is obvious from definitions. To prove the second one, since every open set $E_j$
is contained in a ball $B_j$ of diameter $2\diam(E_j)$, for every covering
$\{E_j\}$ of $A\subseteq M$ with $\diam(E_j) < \delta$ it holds
$$
\sum_{i=1}^{+\infty} \Psi\big( \diam(E_j)\big) \ge \frac{1}{C}
\sum_{i=1}^{+\infty} \Psi\big( 2\diam(E_j)\big) = \frac{1}{C}
\sum_{i=1}^{+\infty} \Psi\big( \diam(B_j)\big).
$$
Now, taking the infimum, in the right hand-side, with respect to all covering
$\{B_j\}$ with balls of diameter less than $2\delta$, and then doing the same on
the left hand side, letting $\delta \downarrow 0$ we deduce the desired
$\overline \haus_\Psi \le C \haus_\Psi$.
}
\end{remark}

\section{The main theorem and its proof}\label{sec_statement}

We are ready to state our main result in its general form.

\begin{theorem}\label{teo_maingenerale}

Let $\varphi \colon M^m \ra D \subset N^n$ be an isometric immersion into a
bounded, $m$-convex open subset $D$ of a Cartan-Hadamard manifold $N$
satisfying $\overline K_\rad \le -B^2$, for some $B \ge 0$. Let $(F, c)$ be the
uniform convexity data of $D$ as in Definition \ref{def_convexity}, and let $R_0
> \frac 12 \mathrm{diam}(D)$. Set
\begin{equation}\label{exmut}
\mu(t) = \left\{ \begin{array}{ll}
t & \quad \text{if } \, B=0, \\[0.1cm]
B^{-1} \tanh(Bt) & \quad \text{if } B>0
\end{array}\right.
\end{equation}
Suppose that the mean curvature $H$ of $\varphi$ satisfies
\begin{equation}\label{ipocurvatmedia}
\|H\| = \|H\|_{L^\infty(M)} < \left\{ \begin{array}{ll}
\disp \min\left\{\frac{m-1}{m\mu(2R_0)} \, , \, \frac{c}{m\|\overline\nabla
F\|_{L^\infty(D)}}\right\} & \quad \text{if } \lim\varphi \cap \partial D \neq
\emptyset, \\[0.5cm]
\disp \frac{m-1}{m\mu(2R_0)} & \quad \text{if } \lim\varphi \cap \partial D =
\emptyset,
\end{array}\right.
\end{equation}
and set
$$
\theta = m - 1 - m \mu(2R_0)\|H\|.
$$
If $\haus_\Psi(\lim \varphi \cap D) = 0$, where
\begin{equation}\label{definpsi}
\begin{array}{ll}
\Psi(t) = t^2 & \quad \text{if } \theta>1 \\[0.1cm]
\Psi(t) = t^2|\log t| & \quad \text{if } \theta=1 \\[0.1cm]
\Psi(t) = t^{\theta+1} & \quad \text{if } \theta \in (0,1),
\end{array}
\end{equation}
then the spectrum of $-\Delta$ is discrete.
\end{theorem}

In the minimal case, the result is particularly transparent and extends Theorem
\ref{teo_maindim2} to the higher-dimensional case.

\begin{theorem}\label{teo_main}

Let $\varphi \colon M^m \hra D \subset N^n$ be a minimal immersion into a
bounded, $m$-convex subset of a Cartan-Hadamard manifold $N$. If
$\haus_\Psi(\lim\varphi \cap D) = 0$, where
\begin{equation}\label{definpsirelaxed}
\Psi(t) = t^2 \quad \text{if } m>2, \qquad \text{or} \qquad \Psi(t) = t^2|\log
t| \quad \text{if } m=2,
\end{equation}
then the spectrum of $-\Delta$ is discrete.
\end{theorem}

\begin{remark}
\emph{We underline that, in Theorem \ref{teo_main}, raising the dimension of the
manifold $M$ does not yield an improvement of the allowed Hausdorff dimension of
$\lim\varphi \cap D$. Indeed, the exponent $2$ is essential for the arguments of
the proof of Theorem \ref{teo_maingenerale} to work. We will come back to this
observation later.
}
\end{remark}

We now come to a brief description of the strategy of the proof. To show that
$-\Delta$ has discrete spectrum, as usual we rely on a combination of Persson
formula, see \cite{persson}, and Barta's inequality (\cite{barta}, and its
generalized version in \cite{bessa-montenegro2}). Persson formula relates the
infimum $\inf\sigma_\ess(-\Delta)$ of the essential spectrum of $-\Delta$ to the
fundamental tone of the complementary of compact sets:
\begin{equation}\label{persson}
\inf\sigma_\ess(-\Delta) = \sup_{\footnotesize K \subset M \text{ compact}}
\lambda^*(M\backslash K)
\end{equation}
where  $\lambda^*(M\backslash K)$ is the bottom of the spectrum of the
Friedrichs extension of $(-\Delta, C^\infty_c(M\backslash K))$. On the other
hand, Barta inequality gives a lower bound for $\lambda^*(M\backslash K)$ via
positive functions:
\begin{equation}\label{barta}
\lambda^*(M\backslash K) \ge \inf_{M\backslash K} \frac{-\Delta w}{w} \qquad
\text{for every }\, 0<w \in C^2(M\backslash K).
\end{equation}
Combining \eqref{persson} and \eqref{barta}, to prove that $-\Delta$ has
discrete spectrum or equivalently, by the min-max theorem, that
$\inf\sigma_\ess(-\Delta) = +\infty$ it is therefore enough to find an
increasing sequence of compact subsets $\{K_l\}$ of $M$, and functions $0<w_l
\in C^2(M\backslash K_l)$, such that
\begin{equation}\label{ipospettro}
\frac{-\Delta w_l}{w_l} \ge c_l \,\,\, \text{on } M\backslash K_l,\,\,\,\,\,
\text{with} \,\,\, c_l \ra +\infty \,\,\, \text{as }\,
 l \ra +\infty.
\end{equation}

Each $w_l$ will be constructed as a sum of suitable positive strictly
superharmonic functions, depending on the balls of a good covering of $\lim
\varphi$. The key point to construct them is the next lemma. We state it in a
rather general form in order to put in evidence the flexibility of the
procedure, and to underline some subtle phenomena.

\subsection*{The fundamental lemma}\label{sec_lemma}

Given a bounded immersion $\varphi \colon M^m \ra N^n$, the next key lemma
enables us to construct bounded, strictly subharmonic functions on $M$ with a
very precise control both on their Laplacian and on their $L^\infty$-norm. We
remark that this is possible since, in our assumptions on $\|H\|$ and on $N$,
$M$ turns out to be non-parabolic.\\
Throughout this section, we assume the following:
\begin{itemize}
\item[$(\mathcal{H}_1)$] $N^n$ has a pole $x_0$, and the radial sectional
curvatures $\overline{K}_\rad$ of radii issuing from $x_0$ in $N$ satisfy
$$
\overline{K}_{\rad}(y) \le - G( \rho(y) ), \qquad \rho(y) =
\mathrm{dist}(y,x_0),
$$
for some $G\in C^0(\R^+_0)$.
\item[$(\mathcal{H}_2)$] The solution $h(t)$ of \eqref{defh} satisfies
\begin{equation}
h,h'>0 \quad \text{on } \R^+, \qquad h(t) \uparrow +\infty \quad \text{as } t
\ra +\infty.
\end{equation}
\end{itemize}

\begin{remark}\label{rem_ipoGh}

\emph{By Proposition 1.21 in \cite{bmr2}, $h$ satisfies $(\mathcal{H}_2)$
whenever the negative part of $G$ is small in the following sense:
\begin{equation}\label{buonacurva}
G_-(s) \le \frac{1}{4s^2} \qquad \text{on } \R^+.
\end{equation}
Furthermore, \eqref{buonacurva} ensures the validity of the lower bound $h(t)
\ge C\sqrt{t}\log t$ for $t \ge t_0\ge 2$, for some positive constant $C$.
}
\end{remark}

We set for convenience $\mu(t) = \|h/h'\|_{L^\infty([0,t])}$. Since $h'(0)=1$,
fix $\bar{a} \in (0,1)$ small enough that
\begin{equation}\label{ipobara}
h'(t) \ge \frac 12 \qquad \text{for every } t \in [0,\bar a].
\end{equation}

\begin{lemma}\label{basicconstr}
Suppose that the conditions $(\mathcal{H}_1)$, $(\mathcal{H}_2)$ are met, and
let $\varphi \colon M^m \ra N^n$ be an isometric immersion into a ball
$B_R(x_{0})\subset N$, with mean curvature satisfying
\begin{equation}\label{ipocurvatmedia}
\Vert H\| = \|H\|_{L^\infty(M)} < \frac{m-1}{m\mu(R)}.
\end{equation}
Set
$$
\theta = m - 1 - m\|H\|\mu(R) >0,
$$
and choose a positive, non-increasing function $S\in C^0(\R^+_0)$ satisfying
\begin{equation}\label{defS}
S(0) =1, \qquad \int_0^{+\infty} t^\theta S(t) \di t = \hat{S} <+\infty.
\end{equation}
Then, there exists a positive constant $C$, depending on $m,R,\theta, S$ and
$h_{|_{[0,R]}}$ such that the following holds: for each $a \in (0,\bar a]$,
there is a smooth function $u_{x_0} \colon M \ra \R$ such that
\begin{eqnarray}
(i) & & u_{x_0} \ge 0, \quad  u_{x_0}(p) =0 \ \text{ if and only if }
\varphi(p)=x_0, \label{propux1}\\[0.2cm]
(ii) & & \|u_{x_0}\|_{L^\infty} \le \left\{\begin{array}{ll} C a^2 & \quad
\text{if } \theta>1 \\[0.1cm]
C a^2|\log a| & \quad \text{if } \theta = 1 \\[0.1cm]
C a^{\theta+1} & \quad \text{if } \theta \in (0,1) \end{array}\right.
\label{propux2} \\[0.3cm]
(iii) & & \Delta u_{x_0} \ge \left\{\begin{array}{ll} \disp \frac{(\theta+1)}{2}
& \quad \text{on } \varphi^{-1}\big\{B_a(x_0)\big\}, \\[0.3cm]
\disp \theta h'(\rho \circ \varphi) S\left(\frac{h(\rho\circ
\varphi)-h(a)}{h(a)}\right) & \quad \text{on } \varphi^{-1}\big\{N\backslash
B_a(x_0)\big\}.
\end{array}\right. \label{propux3}
\end{eqnarray}
\end{lemma}

\begin{remark}
\emph{
Note that $\mu(R)\|H\|$ is scale-invariant, thus $\theta$ defines a genuine
geometrical object associated to the immersion.
}
\end{remark}

\begin{proof}

First of all we recall that, for a function $f \in C^2(N)$, by the chain rule,
the composition $f\circ \varphi \in C^2(M)$ satisfies
$$
\nabla \di (f \circ \varphi) = \overline{\nabla} \di f (\di \varphi, \di
\varphi) + \di f(\nabla \di \varphi),
$$
Now, let $\{e_i, e_\alpha\}$ be a local Darboux frame along $\varphi$, with
$\{e_i\}$ tangent to $M$. Then, tracing the above equality, it yields
\begin{equation}\label{formulalapla}
\Delta(f \circ \varphi) = \disp \sum_{j=1}^m \overline{\nabla} \di f(e_j,e_j) +
m \di f(H).
\end{equation}
Define $f(y)= g(\rho(y))$, for some suitable $g \in C^2(\R^+_0)$, $g'\geq 0$,
that will be chosen in a moment. By the Hessian comparison theorem (see
\cite{bmr2}, Theorem 1.15)
$$
\overline{\nabla} \di \rho \ge \frac{h'(\rho)}{h(\rho)} \Big( \metric - \di \rho
\otimes \di \rho \Big).
$$
Hence, if $g$ is increasing,
\begin{equation}\label{hessiana}
\overline{\nabla} \di f \ge \frac{g'(\rho)h'(\rho)}{h(\rho)} \Big( \metric - \di
\rho \otimes \di \rho \Big) + g''(\rho) \di \rho \otimes \di \rho.
\end{equation}
By \eqref{hessiana}, and using that $|\di \rho|=1$,
\begin{equation}\label{ilboundperdeltaf}
\begin{array}{l}
\disp \sum_{j=1}^m \overline{\nabla} \di f(e_j,e_j) + m\di f(H) = \frac{g'h'}{h}
\Big( m - \sum_{j=1}^m \di \rho(e_j)^2 \Big) + g'' \sum_{j=1}^m \di \rho(e_j)^2
+ mg'\di \rho(H) \\[0.4cm]
\disp \ge \frac{g'h'}{h} \Big( m - \sum_{j=1}^m \di \rho(e_j)^2 -
m\frac{h}{h'}\|H\| \Big) + g''\sum_{j=1}^m \di \rho(e_j)^2 \\[0.4cm]
\disp \ge \frac{g'h'}{h} \Big( m - \sum_{j=1}^m \di \rho(e_j)^2 - m\mu(R)\|H\|
\Big) + g''\sum_{j=1}^m \di \rho(e_j)^2.
\end{array}
\end{equation}
Define
$$
w(t) = \left\{\begin{array}{ll}
\disp (\theta+1)h'(t) & \quad \text{if } t \le a \\[0.2cm]
\disp (\theta+1)h'(t) S\left(\frac{h(t)-h(a)}{h(a)}\right) & \quad \text{if } t
\ge a,
\end{array}\right.
$$
and set
\begin{equation}\label{defing}
g(t) = \int_0^t \frac{1}{h(s)^\theta} \left[\int_0^s h(\sigma)^\theta w(\sigma)
\di \sigma\right] \di s.
\end{equation}
Note that
\begin{equation}\label{gsutlea}
g(t)= \int_0^t h(s) \di s \qquad \text{for } t \le a,
\end{equation}
so $g \in C^2(\R^+_0)$, and that
\begin{equation}\label{eq_g}
g'(t) = \frac{1}{h(t)^\theta} \int_0^t h(s)^\theta w(s) \di s >0, \qquad
\big(h(t)^\theta g'(t)\big)' = h(t)^\theta w(t) \qquad \text{on } \R^+.
\end{equation}
Since, on $[0,a]$, by \eqref{gsutlea} it holds $g''(t)/g'(t)= h'(t)/h(t)$, the
estimate \eqref{ilboundperdeltaf} implies
\begin{equation}\label{primadia}
\sum_{j=1}^m \overline{\nabla} \di f(e_j,e_j) + m\di f(H) \ge \left(m- m
\mu(R)\|H\|\right) g''(\rho) = (\theta+1) h'(\rho)
\end{equation}
at every point $p \in M$ such that $\varphi(p) \in B_a(x_0)$. Consider now a
point $p \in M$ such that $\rho(\varphi(p))\ge a$ and $g''(\rho(\varphi(p))) \ge
0$. Using $|\di \rho|=1$, \eqref{eq_g} and the fact that $g,h$ are increasing we
deduce that,  at $p$,
\begin{equation}\label{g2magzero1}
\begin{array}{lcl}
\disp \sum_{j=1}^m \overline{\nabla} \di f(e_j,e_j) + m\di f(H) & \ge & \disp
\left(m- 1 - m\mu(R)\|H\|\right) \frac{g'(\rho)h'(\rho)}{h(\rho)} \\[0.4cm]
\disp & = & \disp \theta \frac{h'(\rho)}{h(\rho)^{\theta+1}}\int_0^\rho
h(s)^{\theta}w(s)\di s.
\end{array}
\end{equation}
Now, by the definition of $w(s)$ and since $h$ is increasing and $S$ is
non-increasing,
\begin{eqnarray}
\disp \frac{h'(\rho)}{h(\rho)^{\theta+1}}\int_0^\rho h(s)^{\theta}w(s)\di s & =
& \disp \frac{h'(\rho)h(a)^{\theta+1}}{h(\rho)^{\theta+1}}\nonumber  \\[0.4cm]
& & \disp + \frac{h'(\rho)}{h(\rho)^{\theta+1}}\int_a^\rho
h(s)^{\theta}(\theta+1)h'(s)S\left(\frac{h(s)-h(a)}{h(a)}\right)\di s  \\[0.4cm]
& \ge & \!\!\!\disp \frac{h'(\rho)h(a)^{\theta+1}}{h(\rho)^{\theta+1}} +
\frac{h'(\rho)}{h(\rho)^{\theta+1}}S\left(\frac{h(\rho)-h(a)}{h(a)}\right)
\left[ h(\rho)^{\theta+1}-h(a)^{\theta+1}\right] \nonumber\\[0.4cm]
& \ge & \disp h'(\rho) S\left(\frac{h(\rho)-h(a)}{h(a)}\right).\nonumber
\end{eqnarray}
Combining with \eqref{g2magzero1} we get
\begin{equation}\label{g2magzero}
\disp \sum_{j=1}^m \overline{\nabla} \di f(e_j,e_j) + m\di f(H) \ge \theta
h'(\rho) S\left(\frac{h(\rho)-h(a)}{h(a)}\right)
\end{equation}
at those points for which $\rho(\varphi(p)) \ge a$ and $g''(\rho(\varphi(p)))
\ge 0$. On the other hand, at those points for which $g''(\rho(\varphi(p))) \le
0$ we can bound as follows:
\begin{equation}\label{dopodia}
\begin{array}{lcl}
\disp \sum_{j=1}^m \overline{\nabla} \di f(e_j,e_j) + m\di f(H) & \ge & \disp
\frac{g'(\rho)h'(\rho)}{h(\rho)} \Big( m - 1 - m\mu(R)\|H\| \Big) - |g''(\rho)|
\\[0.4cm]
& = & \disp \theta\frac{g'(\rho)h'(\rho)}{h(\rho)} + g''(\rho) \\[0.4cm]
& = & \disp h(\rho)^{-\theta} \big( h(t)^\theta g'(t)\big)'_{|t=\rho} = w(\rho)
\\[0.4cm]
& = & \disp (\theta+1)h'(\rho) S\left(\frac{h(\rho)-h(a)}{h(a)}\right)
\end{array}
\end{equation}
the second equality following from \eqref{eq_g}, and the third from the
definition of $w(s)$. Putting together \eqref{g2magzero} and \eqref{dopodia} we
get
\begin{equation}\label{g2magzero3}
\disp \sum_{j=1}^m \overline{\nabla} \di f(e_j,e_j) + m\di f(H) \ge \theta
h'(\rho) S\left(\frac{h(\rho)-h(a)}{h(a)}\right)
\end{equation}
for every $p$ with $\rho(\varphi(p)) \ge a$. Concluding, from
\eqref{formulalapla}, \eqref{primadia}, \eqref{g2magzero} and \eqref{dopodia} we
get
\begin{equation}\label{bonlapla}
\left\{\begin{array}{ll}
\Delta (f\circ \varphi) \ge (\theta+1) h'(\rho \circ \varphi) & \quad \text{on }
\varphi^{-1}\big\{B_a(x_0)\big\}, \\[0.3cm]
\disp \Delta (f\circ \varphi) \ge \theta h'(\rho \circ \varphi)
S\left(\frac{h(\rho\circ \varphi)-h(a)}{h(a)}\right) & \quad \text{on }
\varphi^{-1}\big\{N\backslash B_a(x_0)\big\}.
\end{array}\right.
\end{equation}
Define $u_{x_0} = f \circ \varphi$. Property \eqref{propux1} is immediate, and
\eqref{propux3} follows from \eqref{ipobara} and \eqref{bonlapla}. We are left
to prove the $L^\infty$-bounds \eqref{propux2}. By \eqref{defS}, there exists $S^*=S^*(\theta,S)$ such that
\begin{equation}\label{propS}
(\theta +1)\sum_{k=1}^{+\infty} S(k)(k+1)^{\theta}  = S^* < \infty.
\end{equation}
Define a sequence
$\{a_k\}_{k=0}^{+\infty}$ in such a way that
\begin{equation}\label{propak}
h(a_k) = kh(a).
\end{equation}
In our assumption $(\mathcal{H}_2)$, $h$ is increasing and $h(t) \ra +\infty$ if
$t \ra +\infty$, thus each $a_k$ is well defined,  $\{a_k\}$ is increasing and
$a_k \ra +\infty$ as $k\ra +\infty$. For every $s \in \R^+$, define $N(s)$ to be
the greatest $k$ such that $a_k < s$. Then, for $s \ge a$,
\begin{equation}\label{gprimostima}
\begin{array}{l}
\disp \int_0^s h(\sigma)^\theta w(\sigma) \di \sigma  = \int_0^a \ldots +
\int_a^s \ldots = h(a)^{\theta+1} + \int_a^s h(\sigma)^\theta w(\sigma) \di
\sigma \\[0.4cm]
\disp \le h(a)^{\theta+1} + \sum_{k=1}^{N(s)} \int_{a_k}^{a_{k+1}} (\theta+1)
h(\sigma)^\theta h'(\sigma)
S \left(\frac{h(\sigma)-h(a)}{h(a)}\right)\di \sigma.
\end{array}
\end{equation}
Since $h$ is increasing, $h \ge h(a_k)$ on $[a_k, a_{k+1}]$. Property
\eqref{propak} and the fact that $S$ is non-increasing thus imply
\begin{equation}\label{gprimostima2}
\begin{array}{l}
\disp \int_a^s h(\sigma)^\theta w(\sigma) \di \sigma  \le \sum_{k=1}^{N(s)} S(k)
\int_{a_k}^{a_{k+1}} (\theta+1) h(\sigma)^\theta h'(\sigma)  \di \sigma
\\[0.4cm]
\disp \le \sum_{k=1}^{N(s)} S(k) \big[h(a_{k+1})^{\theta+1} -
h(a_{k})^{\theta+1}\big] \\[0.4cm]
\disp \le h(a)^{\theta+1}\sum_{k=1}^{N(s)} S(k) \big[(k+1)^{\theta+1} -
k^{\theta+1}\big] \le h(a)^{\theta+1}\sum_{k=1}^{N(s)} (\theta+1)
S(k)(k+1)^{\theta}
\end{array}
\end{equation}
the last inequality following by the mean value theorem. Therefore, putting
together \eqref{defing}, \eqref{gprimostima} and \eqref{gprimostima2} we argue
that, for $t \ge a$,
\begin{equation}\label{belbound}
\begin{array}{lcl}
\disp g(t) & = & \disp g(a) + \int_a^t \frac{1}{h(s)^\theta} \left[ \int_0^s
h(\sigma)^\theta w(\sigma) \di \sigma \right]\di s\\[0.4cm]
& \le & \disp g(a) + \int_a^t \frac{1}{h(s)^\theta} \left[ h(a)^{\theta+1} +
h(a)^{\theta+1}\sum_{k=1}^{N(s)} (\theta+1)S(k) (k+1)^{\theta}\right] \di
s\\[0.4cm]
& \le & \disp g(a) +  (1+S^*)h(a)^{\theta+1} \int_a^t \frac{\di s}{h(s)^\theta},
\end{array}
\end{equation}
where the last inequality comes from the definition of $S^*$ in \eqref{propS}.
Combining with \eqref{primadia} we conclude that
\begin{equation}\label{superg}
\begin{array}{lcl}
\disp \|g\|_{L^\infty([0,R])} & = &  \disp g(R) \le  g(a) + (1+S^*)
h(a)^{\theta+1} \int_a^{R} \frac{\di s}{h(s)^\theta} \\[0.4cm]
& = &\disp  \int_0^a h(s) \di s + (1+S^*) h(a)^{\theta+1} \int_a^{R} \frac{\di
s}{h(s)^\theta}
\end{array}
\end{equation}

From $h(s) = s + O(s^2)$ as $s\ra 0$ it is easy to deduce that
\begin{equation}\label{asinto}
\int_0^ah(s)\di s = \frac{a^2}{2} + O(a^3), \qquad \int_a^{R} \frac{\di
s}{h(s)^\theta} \sim
\left\{ \begin{array}{ll}
\disp \frac{a^{1-\theta}}{\theta-1} & \quad \text{if } \theta>1 \\[0.3cm]
\disp |\log a| & \quad \text{if } \theta=1, \\[0.3cm]
\disp C & \quad \text{if } \theta \in (0,1), \\[0.3cm]
\end{array}\right.
\end{equation}
For some $C$ depending on $m,R,\theta$ and $h$ on $[0,R]$. From \eqref{superg}
and \eqref{asinto}, there exists a positive constant $C$ depending only on
$m,R,\theta,S$ and on $h_{|_{[0,R]}}$ such that, if if $a \in (0, \bar a]$,
$$
\|g\|_{L^\infty([0,R])} \le \left\{ \begin{array}{ll} C a^2 & \quad \text{if }
\theta>1 \\[0.2cm]
a^2 + C a^2|\log a| \le C a^2 |\log a| & \quad \text{if } \theta=1 \\[0.2cm]
a^2 + C a^{\theta+1} \le Ca^{\theta+1} & \quad \text{if } \theta \in (0,1).
\end{array}\right.
$$
Noting that
$$
\|u_{x_0}\|_{L^\infty(M)} \le \|f\|_{L^\infty(B_{R}(x))} =
\|g\|_{L^\infty([0,R])},
$$
the desired bounds in \eqref{propux2} are proved.
\end{proof}

\begin{remark}\label{rem_parabolicity}

\emph{If $M$ is minimal then,  the constant $C$ in \eqref{propux2} does
\emph{not} depend on $R$ whenever the non-parabolicity condition
\begin{equation}\label{nonparabol}
\frac{1}{h(s)^{m-1}} \in L^1(+\infty)
\end{equation}
holds. Indeed, if $M$ is minimal then $\theta = m-1$. Once $S$ is chosen, the
value $S^*$ in \eqref{propS} is thus independent of $R$. By \eqref{nonparabol},
in the bound \eqref{superg} we can let $R\ra +\infty$ to obtain
\begin{equation}
\disp \|g\|_{L^\infty(\R^+_0)} \le \int_0^a h(s) \di s + (1+S^*) h(a)^{m}
\int_a^{+\infty} \frac{\di s}{h(s)^{m-1}}.
\end{equation}
The asymptotics $h(a) \sim a$ and
$$
h(a)^{m} \int_a^{+\infty} \frac{\di s}{h(s)^{m-1}} \sim \left\{
\begin{array}{ll}
\frac{a^2}{m-2} & \quad \text{if } m \ge 3, \\[0.2cm]
a^2 |\log a| & \quad \text{if } m=2
\end{array}\right.
$$
prove that $\|g\|_{L^\infty(\R^+_0)} \le Ca^2$ if $m\ge 3$ (respectively,
$\|g\|_{L^\infty(\R^+_0)} \le Ca^2|\log a|$ if $m =2$), for some constant $C$
only depending on $m,S^*$, proving the claim.}
\vspace{2mm}

\emph{By Remark \ref{rem_ipoGh}, condition \eqref{nonparabol} is always
satisfied whenever \eqref{buonacurva} holds and $m \ge 3$. We further observe
that, via Proposition 3.1 in \cite{Gr}, requirement \eqref{nonparabol} is
necessary and sufficient for the non-parabolicity of the radially symmetric
model $M^m_h$, which is, roughly speaking, compared to $M$ by means of formulas
\eqref{hessiana}. We recall that $M^m_h$ is defined as the manifold $\R^m$, with
a fixed origin $o$ and metric given, in polar geodesic coordinates $(r,\theta)$
centered at $o$, by  $\di s_h^2 = \di r^2 + h(r)^2 g_{\esse^{m-1}}(\theta)$,
where $g_{\esse^{m-1}}(\theta)$ is the standard metric on the unit sphere.
Clearly, since we are constructing bounded, non-constant and strictly
subharmonic functions on $M$, some non-parabolicity condition must necessarily
appear.
}
\end{remark}

\subsection*{Proof of Theorem \ref{teo_maingenerale}}

Property $R_0> \frac 12\mathrm{diam}(D)$ and the completeness of $N$ enable us
to choose $x \in N$ such that $D \Subset B_{R_0} = B_{R_0}(x)$. For notational
convenience, define $\lim_0\varphi = \lim\varphi \cap D$. Choose a small $r_1<<
\min\{R_0,1\}$ in such a way that
$$
B_{2r_1}(\mathrm{lim}_0\, \varphi) \subseteq B_{R_0}.
$$
Since the function $\Psi$ defined in \eqref{definpsi} satisfies \eqref{propper},
by Remark \ref{rem_haus} the measure $\overline \haus_\Psi(\lim_0\varphi)$
computed by only using balls is zero. Therefore, we can find a countable
covering $\{B_j\}$ of balls $B_j = B_{\eps_j}(x_j) \subseteq M$ of radius
$\eps_j \le r_1$ such that
\begin{equation}\label{propcovering}
\mathrm{lim}_0\,\varphi \subseteq \bigcup_j B_j \qquad \text{and} \qquad \left|
\sum_j \Psi(\eps_j)\right| \le r_1.
\end{equation}
For notational convenience, set $b_1= 8\sqrt{r_1}$. Consider $\lim_1\varphi =
\lim\varphi \cap \big(D\backslash B_{b_1/2}(\partial D) \big)$. This set is
compact and contained in $\lim_0\varphi$. By compactness, we can select a finite
subcovering $\{B_j\}_{j=1}^{k_1}$ of balls touching $\lim_1\varphi$, in such a
way that the covering is contained in $B_{2r_1}(\lim_0\varphi)\subseteq
B_{R_0}$. We can suppose $k_1 \ge 2$. Clearly, the second property in
\eqref{propcovering} still holds for the subcovering.
Define the compact set
$$
K_{r_1} = M \backslash \left(\varphi^{-1}\Big(\bigcup_{j=1}^{k_1} B_j\Big) \cup
\varphi^{-1}\big( B_{b_1}(\partial D)\big) \right)
$$
For each $j$, we choose $a_j=\eps_j\le r_1$. Since $x_j \in B_{R_0}$, we can
apply Lemma \ref{basicconstr} with the choices $G=B^2$ (so that $h(s)=s$ if
$B=0$, $h(s) = B^{-1} \sinh(Bs)$ if $B>0$), $x=x_j$, $R=2R_0$ and $S(t) =
\max\{t,1\}^{-\theta -2}$ to deduce the existence of constants $\bar a, C>0$
($C$ only depending on $m,R,\theta$ and on $h_{|_{[0,2R_0]}}$) such that, if
$r_1\le \bar a$, there exists $u_j = u_{x_j}$ with the properties
\begin{equation}\label{propux3b}
\left\{\begin{array}{l}
u_j \ge 0, \quad  u_j(p) =0 \ \text{ if and only if } \varphi(p)=x_j, \\[0.2cm]
\|u_j\|_{L^\infty} \le  C \Psi(\eps_j) \\[0.3cm]
\Delta u_j > 0 \ \text{ on } M, \quad \Delta u_j \ge (\theta+1)/2 \ \text{ on }
\varphi^{-1}(B_j).
\end{array}\right.
\end{equation}
We consider the case $\lim \varphi \cap \partial D \neq \emptyset$ in
\eqref{ipocurvatmedia}, the other case being easier. Define $u_\infty =
b_1F(\varphi(x))$. Then, $u_\infty \le 0$ on $M$. By formula
\eqref{formulalapla} with $f=F$ and the $m$-convexity of $D$ with constant $c$
we deduce that, on the whole $M$,
\begin{equation}\label{laparteconF}
\begin{array}{lcl}
\Delta u_\infty & = & \disp b_1\Delta (F \circ \varphi) \ge
b_1\left(\sum_{j=1}^m \overline \nabla \di F(e_j, e_j) - m\|\overline \nabla
F\|_{L^\infty(D)}\|H\| \right)\\[0.4cm]
& \ge & b_1\Big(\disp c - m\|\overline \nabla F\|_{L^\infty(D)}\|H\|\Big) \ge
b_1 \overline C,
\end{array}
\end{equation}
for some positive constant $\overline C$. Set
$$
w_1 = \sum_{j=1}^{k_1} (2\|u_j\|_{L^\infty}-u_j) - u_\infty.
$$
Note that, since $k_1 \ge 2$, $w_1$ is strictly positive on $M$ by construction.
For every $p \in M\backslash K_{r_1}$, either
$$
(1) \ \ \ \ \varphi(p) \in B_{b_1}(\partial D) \qquad \text{or} \qquad (2) \ \ \
\ \varphi(p) \in \bigcup_{j=1}^{k_1} B_j,
$$
and the two cases are not mutually excluding. We first examine $(1)$. In this case,
since $F = 0$ on $\partial D$ and is Lipschitz on $D$,
$$
|u_\infty(p)| \le b_1^2 \|\overline \nabla F\|_{L^\infty(D)},
$$
whence by \eqref{propcovering}, \eqref{propux3b} and the fact that $b_1^2 =
64r_1$ we obtain
\begin{equation}\label{caso1}
\begin{array}{lcl}
\disp -\frac{\Delta w_1}{w_1}(p) & \ge & \disp \frac{\Delta
u_\infty(p)}{2\left(\sum_j\|u_j\|_{L^\infty}\right)+ |u_\infty(p)|} \ge
\frac{b_1 \overline C}{C\left(\sum_j \Psi(\eps_j)\right) + b_1^2\|\overline
\nabla F\|_{L^\infty(D)}} \\[0.6cm]
& \ge & \disp C\frac{b_1}{r_1 + b_1^2} \ge \frac{C}{b_1} = \frac{C}{\sqrt{r_1}}
\end{array}
\end{equation}
for some constant $C>0$ depending on $m, R_0, \theta, F, h_{|_{[0,2R_0]}}$
(hereafter, we will say that it depends on the data of the immersion) but not on
$r_1$ or on the covering $\{B_j\}$, and that may vary from line to line in
\eqref{caso1}.\\
We now examine case $(2)$. We split $\{1,\ldots, k_1\}$ in two subsets
$$
J_1^p = \big\{ j \in \{1,\ldots, k_1\} : \varphi(p) \in B_j \big\}, \qquad J_2^p
= \big\{ j \in \{1,\ldots, k_1\} : \varphi(p) \in N \backslash B_j \big\}.
$$
By construction and by $(2)$, $J_1^p \neq \emptyset$. Then, again by
\eqref{propcovering}, \eqref{propux3b}, and since $b_1 = 8\sqrt{r_1} \ge 8r_1$
by our choice of $r_1$, it holds
\begin{equation}\label{caso2}
\begin{array}{lcl}
\disp -\frac{\Delta w_1}{w_1}(p) & \ge & \disp \frac{\sum_{J^p_1 \cup J^p_2}
\Delta u_j(p)}{2\left(\sum_j\|u_j\|_{L^\infty}\right) + \|u_\infty\|_{L^\infty}}
\ge \frac{\sum_{J^p_1} \Delta u_j(p)}{C\left(\sum_j \Psi(\eps_j)\right) + b_1
\|F\|_{L^\infty(D)}} \\[0.6cm]
& \ge & \disp \frac{(\theta+1)|J^p_1|}{C (r_1 + b_1)} \ge \frac{C}{b_1} =
\frac{C}{\sqrt{r_1}}
\end{array}
\end{equation}
for some $C>0$ depending, as before, on the data of the immersion but not on
$r_1$ or on the covering $\{B_j\}$. Summarizing, there exists $C>0$ only
depending on the data of the immersion such that
\begin{equation}\label{primostep}
- \frac{\Delta w_1}{w_1} \ge \frac{C}{\sqrt{r_1}} \qquad \text{on } K_{r_1}.
\end{equation}
Now, choose a positive $r_2< \min\{\frac{r_1}{2}, \frac{b_1}{16}, \frac 12\}$
and set $b_2= 8\sqrt{r_2}$.

Again, we can find a countable covering $\{B_j\}$ of balls $B_j =
B_{\eps_j}(x_j) \subseteq M$ of radius $\eps_j \le r_2$ such that
\begin{equation}
\mathrm{lim}_0\,\varphi \subseteq \bigcup_j B_j \qquad \text{and} \qquad \left|
\sum_j \Psi(\eps_j)\right| \le r_2.
\end{equation}
Consider the compact set $\lim_2\varphi = \lim\varphi \cap \big(D\backslash
B_{b_2/2}(\partial D) \big)$, contained in $\lim_0\varphi$ and containing
$\lim_1\varphi$. By compactness, we can select a finite subcovering
$\{B_j\}_{j=1}^{k_2}$ of balls touching $\lim_2\varphi$, so that the subcovering
is contained in $B_{2r_2}(\lim_2\varphi)$. Since $b_2 \ge 8\sqrt{r_2} \ge 8r_2$,
$B_{2r_2}(\lim_2\varphi)\subseteq D$. Moreover, by our choice of $r_2$,
$$
B_{2r_2}(\mathrm{lim}_2\varphi) \Subset B_{2r_1}(\mathrm{lim}_1\varphi) \cup
B_{b_1}(\partial D).
$$
Consequently, the compact set
$$
K_{r_2} = M \backslash \left(\varphi^{-1}\Big(\bigcup_{j=1}^{k_2} B_j\Big) \cup
\varphi^{-1}\big( B_{b_2}(\partial D)\big) \right)
$$
satisfies $K_{r_1} \subseteq \mathrm{int}(K_{r_2})$. The same construction
procedure as above can be repeated verbatim, yielding a superharmonic function
$w_2$ on $M$ such that
$$
- \frac{\Delta w_2}{w_2} \ge \frac{C}{\sqrt{r_2}} \qquad \text{on } M \backslash
K_{r_2},
$$
for the same constant $C$ as in \eqref{primostep}, only depending on the data of
the immersion. Inductively, if at each step we select a positive $r_{l+1}<
\min\{\frac{r_l}{2}, \frac{b_l}{16}, 2^{-l}\}$ and proceed to find $K_{r_{l+1}}$
satisfying $\mathrm{int} (K_{r_{l+1}}) \supseteq K_{r_l}$, and a positive
$w_{l+1}$ solving $-\Delta w_{l+1}/w_{l+1} \ge C/\sqrt{r_{l+1}}$ on $M\backslash
K_{r_{l+1}}$. Note that, although $r_l \downarrow 0^+$, we cannot infer that
$K_{r_l}$ is an exhaustion of $M$ because $\lim\varphi$ could actually contain
points of $\varphi(M)$. However, by Persson formula and Barta inequality, for
each $l$
$$
\inf \sigma_\ess(-\Delta) \ge \lambda^*(M\backslash K_{r_l}) \ge
\inf_{M\backslash K_{r_l}} \left(-\frac{\Delta w_l}{w_l}\right) \ge
\frac{C}{\sqrt{r_l}} \ra +\infty
$$
as $l\ra +\infty$, which concludes the proof.

\section{Remarks on $\sigma_\ess(-\Delta)$ and open
questions}\label{sec_sharpness}

We conclude this paper with some observations on the links between $\sigma_\ess
(-\Delta)$ and the topology of the limit set of an immersion, and with some open
problems. First, we prove Proposition \ref{lem_crite} of the Introduction, that is, that the ball property in Definition \ref{def_ball} implies the existence of essential spectrum.


%

\begin{proof}[Proof of Proposition \ref{lem_crite}]
For each $j$, define the compactly supported, Lipschitz function $\varphi_j(x)=
\psi(\rho_j(x))$, where $\rho_j(x) = \mathrm{dist}(x, x_j)$ and
\begin{equation}
 \psi(t)=\left\{\begin{array}{ll}
                     1,& t \leq \delta R\\[0.1cm]
                     \frac{R - t}{R(1-\delta)} & t \in \left[\delta R,R\right]
\\[0.1cm]
                     0 & t \ge R
                    \end{array}\right. \quad  \Longrightarrow \quad |\psi'| \le \frac{1}{R(1-\delta)}.
\end{equation}

Then, by the ball property \eqref{desig1},
\begin{eqnarray}
I_\lambda(\phi_j,\phi_j) &=&\int_{B_R(x_j)}
|\nabla \phi_j|^2-\lambda\int_{B_R(x_j)}\phi_j^2 \leq
\frac{\vol(B_j)}{R^2(1-\delta)^2} - \lambda\text{vol}
\big(B_{\delta R}(x_j)\big) \\[0.2cm]
&\leq & \vol(B_R) \left( \frac{1}{R^2(1-\delta)^2} - \lambda C^{-1}\right) <0
\end{eqnarray}
provided that $\lambda > \lambda^*= \frac{C}{R^2(1-\delta)^2}$. Since $\{\phi_j\}$ span an
infinite-dimensional subspace of the domain of $-\Delta$, the (Friedrichs
extension of) the operator $-\Delta -\lambda$ has infinite index, or
equivalently $-\Delta$ has infinite eigenvalues below $\lambda$, for each
$\lambda >\lambda^*$. By the min-max theorem (see for instance
\cite{reedsimon3}, or Section 3 in \cite{prs} for a concise account), the
inequality $\inf\sigma_\ess(-\Delta) \le \lambda^*$ follows at once.
\end{proof}

\begin{remark}\label{rem_bishopgromov}
\emph{In virtue of Bishop-Gromov volume comparison theorem, the ball property in Definition \ref{def_ball} is met, for instance, when the Ricci curvature of $B_R(x_j)$ is
uniformly bounded from below by a constant, say $-(m-1)H^2<0$, $m= \dim M$.
Indeed, denoting with $\vol_H(r)$ the volume of a geodesic ball of radius $r$ in
the hyperbolic space $M^m_H$ of sectional curvature $-H^2$, by Bishop-Gromov
theorem $\vol(B_r(x_j))/\vol_H(r)$ is non-increasing on $[0,R]$. Hence, for each
$\delta >0$
$$
\vol\big(B_{\delta R}(x_j)\big) \ge \frac{\vol_H(\delta R)}{\vol_H(R)}
\vol\big(B_R(x_j)\big) = C(\delta,R)^{-1} \vol\big(B_R(x_j)\big).
$$
}
\end{remark}

As already underlined, the independence of Proposition \ref{lem_crite} from curvature requirements makes it
suited to investigate minimal surfaces, particularly those described in the
Introduction. The reason is that the approach via Runge's approximation theorem,
shared by all these constructions, guarantees a $C^0$-control of the metric
(hence, of lengths and volumes) in some known regions, while it seems more
difficult to control curvatures in the same region. We now come to the proof of Proposition \ref{prop_jx}.

\begin{proof}[Proof of Proposition \ref{prop_jx}]
Let $K_n$ the compact sets used in the construction of
\cite{jorge-xavier-annals}, let $r_n$ denote the Euclidean distance between the
inner and the outer circle of $K_n$, and let  $p_n\in K_n$ be the point in the
middle of the segment of the real axis crossed by $K_n$. By construction, $r_n
\ra 0^+$ as $n \ra +\infty$, and $\di s^2$ is $C^0$-close to a multiple of the
Euclidean metric $\di s^2_E$ on $K_n$, more precisely $\di s^2 = \lambda^2 \di
s^2_E$ with
$$
\lambda =\frac 12 \left( |e^h|+ |e^{-h}|\right), \quad \text{for some
holomorphic $h$ on $K_n$ with }  |h - c_n| < 1 \ \text{ on } K_n,
$$
$c_n$ being some positive constant chosen in such a way that
\begin{equation}\label{condicn}
\sum_{n=1, \, n \text{ even}}^{+\infty} r_n e^{c_n-1} = +\infty, \qquad
\sum_{n=1, \, n \text{ odd}}^{+\infty} r_n e^{c_n-1} = +\infty.
\end{equation}
The line elements thus satisfy $C_n \di s_E \le \di s \le e^2 C_n \di s_E$,
where $C_n= (e^{c_n-1} + e^{-c_n-1})/2$. Consequently, every curve passing $K_n$
from the inner to the outer circle has length at least $r_ne^{c_n-1}/2$. The
choice $c_n = r_n$ guarantees both \eqref{condicn} and the following property
that, for each fixed $R>0$, one can find $n_R$ large enough such that
$B_R(p_n)\subseteq K_n$ for each $n \ge n_R$. By the relation of the line
elements, denoting with $\Bo_\rho(z_j)$ the Euclidean ball of radius $\rho$
centered at $z_j$, $\Bo_{Re^{-2}/C_n}(z_j) \subseteq B_R(z_j) \subseteq
\Bo_{R/C_n}(z_j)$, thus
\begin{equation}\label{compa}
\vol\big(B_{\frac R2}(p_n)\big) \ge C_n^2
\vol_E\big(\Bo_{\frac{R}{e^{2}2C_n}}(p_n)\big), \quad \vol\big(B_{R}(p_n)\big)
\le e^4 C_n^2 \vol_E\big(\Bo_{\frac{R}{C_n}}(p_n)\big).
\end{equation}
By the doubling property of Euclidean space $\R^3$, it is immediate from
\eqref{compa} that \eqref{desig1} holds for $\di s^2$ with $\delta=1/2$ and a
suitable $C$ independent of $n$, but even of $R$. Therefore, by Proposition
\ref{lem_crite}, there exists an absolute constant $C>0$ such that
$\inf\sigma_\ess(-\Delta) \le C/R^2$. Since $R$ can be chosen to be arbitrarily
large, the thesis follows.
\end{proof}

Now, let $\varphi : M \ra N$ be an isometric immersion. The ball property \eqref{desig1}
can also be deduced from suitable convergences of pieces of $\varphi(M)$ in $N$.
We formalize this fact in the next

\begin{definition}\label{def_disk}
Let $\varphi : M^m \ra N^n$ be an isometric immersion with $\lim\varphi \neq
\emptyset$. We will say that $\varphi$ has the \emph{\textbf{extrinsic ball property}} if there is a point
$p\in \lim \varphi$, a bounded subset with compact closure $\Sigma_p\subset\lim
\varphi$ containing $p$ and diffeomorphic to the unit ball $B_1(0) \subseteq
\R^m$, and a sequence of disjoint balls
$B_j=B_{R}(x_j)\subseteq M^m$ such that $\varphi(B_j)\to\Sigma_p$ in the
$C^0$-norm.
\end{definition}

\begin{corollary}\label{cor_disk}

If $\varphi : M^m \ra N^n$ has the extrinsic ball property, then $M$ has the ball property. In particular, $\sigma_\ess(-\Delta)$ is
non-empty.
\end{corollary}

\begin{proof}

Let $c>0$ be such that the curvature tensor $R_\Sigma$ of $\Sigma_p$ satisfies
$R_\Sigma \ge -c$. Then, by Bishop-Gromov theorem, there exists $C=C(R)$ such
that $\vol(B_{R/2}(x)) \ge C \vol(B_R(x))$. By $C^0$-convergence, there exists
$j_0$ such that
$$
\vol\big(B_{R/2}(x_j)\big) \ge \frac C2 \vol\big(B_R(x_j)\big) \qquad \text{for
each } j \ge j_0,
$$
thus $M$ has the ball property, and non-empty essential spectrum thanks to Proposition \ref{lem_crite}.
\end{proof}

We underline that, if $\varphi : M^m \ra N^n$ has the extrinsic ball property, then
$\dim_\haus(\lim\varphi) \ge m$. The extrinsic ball property is closely related to the
existence of immersions into $\lim \varphi$ that are generated by $\varphi$
itself. Indeed, we report here the following well-known convergence result:

\begin{lemma}
 Let $M^m_j$ a sequence of complete smooth manifolds, all with injectivity
radius $\mathrm{inj}(M_j) \ge \epsilon_0>0$ and isometrically immersed in a
complete
manifold $N^n$ via $\varphi_j : M_j \ra N$. Let $\Omega_1 \Subset \Omega_2
\Subset N$ be two open, relatively compact sets of $N$. Assume that $\Omega_1$
intersects all $\varphi_j(M_j)$ and that
the second fundamental forms of all $\varphi_j(M_j)\cap \Omega_2$ are uniformly
bounded. Then, there exists a sufficiently small $R>0$ such that
\begin{enumerate}
 \item If $\{x_j\}$, $x_j \in M_j$ is such that $p_j= \varphi_j(x_j) \in
\Omega_1,\; p_j\to p$ and $\di \varphi_j(T_{x_j}M_j)\to \Pi \subseteq T_pN$,
then $\varphi_j$ restricted to the ball $\Sigma_j=B_R(x_j)$ is an embedding such
that $\varphi_j(\Sigma_j)$ converges in the $C^\infty$ topology to an embedded
$m$-ball $\Sigma\ni p$, and $T_p\Sigma=\Pi$.
\item If $M_j=M$ and $\varphi_j=\varphi$ for every $j$, and if $p\in\lim
\varphi$, $M$ has the extrinsic ball property and we can take
$\Sigma_p\subseteq\lim\varphi$.
\item In addition, if each $M_j$ is a complete, simply connected
manifold and the norms of the second fundamental forms $II_j$ of $\varphi_j$
satisfy the following property:
$$
\forall \, R>0, \ \exists \, C_R>0 \, \text{ such that } \,
\left|(II_j)_{\varphi_j^{-1}(B_R(p))}\right|^2 \le C_R \, \text{ for each } j,
$$
then there is an isometric immersion of a complete, simply
connected manifold $M_p$ into $N$ build by convergence over compact sets. In
the case when $M_j=M$, $\varphi_j=\varphi$ for each $j$, and $p \in
\lim\varphi$, we have $M_p\subseteq \lim\varphi$.
\end{enumerate}

\end{lemma}

\begin{remark}
\emph{One idea to prove this result is to find $\eps'>0$ such that each
$\varphi_j(M_j)$ can be written as the image of a section of the normal bundle
over the tangent plane at some point $p_j \in \varphi_j(M_j) \cap \Omega_2$ via
the normal exponential map. Then, use the convergence and regularity theory for
the system of elliptic PDEs describing the mean curvature vector. Simply
connectedness is
important to avoid period problems.
}
\end{remark}

A particular case when the extrinsic ball property holds is when $\varphi$ has locally bounded geometry, in the sense of the following

\begin{definition}
An isometric immersion $\varphi:M^m\to N^n$ has locally bounded geometry if,
for each compact set $K\subseteq N$, then $\varphi|_{\varphi^{-1}(K)}$ has
bounded second fundamental form.
\end{definition}

\begin{corollary}
Let $\varphi:M^m\to N^n$ be an isometric immersion between complete manifolds.
If
$\lim\varphi\neq \emptyset$, $\varphi$ has locally bounded geometry and
$\mathrm{inj}(M) >0$, then $\varphi$ has the extrinsic ball property. In particular, $M$
has non-empty essential spectrum.
\end{corollary}

We conclude by exhibiting an example of a bounded, minimal surface whose
essential spectrum is non-empty, whose limit set has big Hausdorff dimension and
which is not a covering. This, again, shows the sharpness of Theorem
\ref{teo_maindim2}, and introduces us to some open questions.

\begin{example}\label{ex_andrade}

\emph{We consider a portion of Andrade minimal surface, \cite{andrade}. We
recall its construction and some of its properties, referring the reader to
\cite{andrade} for full proofs. Choose $r_1, r_2 >0$ such that $r_1/r_2$ is
irrational and strictly less than $1$, and set $d = r_2-r_1$. Define the map
$\chi: \C \ra \R^3 = \C \times \R$, $\chi(z) = (L(z)- \overline{H(z)}, h(z))$,
for the following choice of holomorphic functions $L,H$ and harmonic function
$h$:
$$
L(z) = (r_1-r_2)e^z, \quad H(z)= -de^{\left(\frac{r_1}{r_2}-1\right)z}, \quad
h(z) = 4\left(\frac{d}{r_2}\right)^{1/2} \left|\frac{r_2}{r_1}\right||r_2-r_1|
\Re\left( ie^{\frac{r_1}{2r_2}z}\right),
$$
where $\Re$ means the real part. Then, a computation gives that
$$
|L'(z)| + |H'(z)| >0, \qquad L'H' = \left( \frac{\partial h}{\partial
z}\right)^2 \qquad \text{on } \C,
$$
which is a necessary and sufficient set of condition on $\chi$ to be a conformal
minimal immersion of $\C$ in $\R^3$. Restricting $\chi$ to the region $U = \{z=
u+iv \in \C : |u| <1\}$, we get a bounded, simply-connected minimal immersion
$\varphi = \chi_{|U}$. For each fixed $u \in (-1,1)$, $\varphi(u+iv)$ is a dense
immersed trochoid in the cylinder $\Gamma_u = \left[ B_{s_1(u)}\backslash
B_{s_2(u)} \right] \times (-l(u), l(u))$, where $s_1, s_2, l$ are explicit
functions of $u$ depending on $r_1$ and $r_2$.
Therefore, $\lim\varphi$ is dense in the open subset $\bigcup_{u \in (-1,1)}
\Gamma_u$ of $\R^3$, which gives $\dim_\haus(\lim\varphi)=3$. Moreover, the
induced metric $\di s^2$ satisfies
\begin{equation}\label{indu}
\di s^2 = \left( |L'|+|H'|\right)^2 |\di z|^2 = \left( |r_2-r_1|e^u + de^{\left(
\frac{r_1}{r_2}-1\right)u}\right)^2|\di z|^2 \ge 4(r_2-r_1)^2 |\di z|^2.
\end{equation}
Considering $z_k = 2ik \in U$, each of the unit balls $\mathbb{B}_1(z_k)
\subseteq U$ in the metric $|\di z|^2$ contains a ball $B_R(z_k)$ in the metric
$\di s^2$ of radius at least $R=2|r_2-r_1|$. Since the sectional curvature of
$\chi$ satisfies
$$
K = -c_1\left(e^{\left(1-\frac{r_1}{4r_2}\right)u} +
c_2e^{\left(\frac{3r_1}{4r_2}-1\right)u}\right)^{-4},
$$
for some positive constants $c_1,c_2$, and $1-\frac{r_1}{4r_2}$ and
$\frac{3r_1}{4r_2}-1$ have opposite signs, $\chi$ has globally bounded
curvature. In particular, $\{B_R(z_k)\}$ is a collection of disjoint balls in
$(U, \di s^2)$ with uniformly bounded sectional curvature, thus
$\sigma_\ess(-\Delta)$ on $(U, \di s^2)$ is non-empty by Proposition
\ref{lem_crite} and Remark \ref{rem_bishopgromov}
.}
\end{example}

\subsection*{Open problems}

\begin{itemize}
\item[$(1)$] As Theorem \ref{teo_main} shows, raising the dimension of $M$ does
not yield an improvement of the allowed Hausdorff dimension of the limit set.
However, from some point of view this fact seems of technical nature. It seems
to us reasonable to state the following
\begin{conjecture}
Let $\varphi : M^m \ra N^{n}$ a minimally immersed submanifold of dimension $m
\ge 3$ into an open, $m$-convex subset $D$ of a Cartan-Hadamard manifold $N$.
If $\haus^{m}(\lim\varphi \cap D) = 0$, then $-\Delta$ has discrete spectrum on
$M$.
\end{conjecture}
\item[$(2)$] Although Theorem \ref{teo_maindim2} is suited for each of the
examples $(i), \ldots, (v)$ in the Introduction, as well as to deal with
solutions of general Plateau problems, it is still unapplicable for the original
example of Nadirashvili in \cite{Nadirashvili}. The reason is that it seems hard
to deduce the behaviour of the limit set from the original construction. Could
it be possible that, for some choice of the parameters in Nadirashvili's
construction, the spectrum of the resulting minimal surface is discrete?
\item[$(3)$] As we have seen in the Introduction, infinite sheet coverings of
complete bounded minimal surfaces always have non-empty essential spectrum. On
the other hand, Example \ref{ex_andrade} establishes the existence of incomplete
minimal surfaces with $\sigma_\ess(-\Delta) \neq \emptyset$ and whose immersion
map $\varphi$ cannot factorize via some Riemannian covering. One could naturally
ask the following
\begin{question}
Is it possible to find a complete, bounded minimal surface $\varphi : M \ra
\R^3$ with non-empty essential spectrum  and such that $\varphi$ cannot
factorize via a Riemannian covering map?
\end{question}

\end{itemize}

\noindent \textbf{Acknowledgements:} this paper has been completed while the third author was enjoying the hospitality of the Universidade Federal do Cear\'{a}-Brazil. The authors are partially supported by  PRONEX/FUNCAP/CNPq. The third author is indebted to Andrea
Carlo Mennucci and Carlo Mantegazza for useful e-mail discussions about the
regularity of distance functions. He also wants to thank Matteo Novaga for his
kind willingness.



\begin{thebibliography}{abc}
\bibitem{alarcon-Ferrer-Martin-GAFA} A. Alarc\'{o}n, L. Ferrer and F. Martin,
{\em Density theorems for complete minimal surfaces in $\mathbb{R}^{3}$.} Geom.
Funct. Anal. \textbf{18} (2008), no. 1, 1--49.
\bibitem{andrade} P. Andrade, \emph{A wild minimal plane in {$\mathbb{R}^3$}}.
Proc. Amer. Math. Soc. \textbf{128} (2000), no. 5, 1451--1457.

\bibitem{barta}J. Barta,  {\em Sur la vibration fundamentale
d'une membrane.} C. R. Acad. Sci. \textbf{204} (1937), 472--473.
\bibitem{bessa-montenegro2}G. P.  Bessa \and J. F. Montenegro, {\em An extension
of Barta's Theorem and geometric applications. }
Ann.  Global Anal. and Geom. \textbf{ 31} (2007),
345--362.

\bibitem{bessa-Jorge-montenegro-JGA} G. Pacelli Bessa, L. Jorge, J. Fabio
Montenegro, {\em The Spectrum of the Martin-Morales-Nadirashvili Minimal
Surfaces is Discrete}. J. Geom. Anal. \textbf{20} (2010),  63--71.
\bibitem{bmr2} B. Bianchini and L. Mari and M. Rigoli, {\em On some aspects of
Oscillation Theory and Geometry,} to appear on Mem. Amer. Math. Soc.

\bibitem{calabi2}E. Calabi, {\em Problems in Differential Geometry} (S.
Kobayashi and J. Eells, Jr., eds.) Proc. of the United States-Japan Seminar in
Differential Geometry, Kyoto, Japan, 1965, Nippon Hyoronsha Co. Ltd., Tokyo
(1966) 170.
\bibitem{chern} S.S. Chern, {\em  The geometry of G-structures.} Bull. Amer.
Math. Soc. \textbf{72} (1966) 167--219.
\bibitem{colding-minicozzi} T. H. Colding, W. P. Minicozzi, {\em  The Calabi-Yau
conjectures for embedded surfaces.} Ann. of Math. (2) \textbf{167} (2008), no.
1, 211--243.
\bibitem{collinrosenberg} P. Collin, H. Rosenberg, \emph{Notes sur la
d\'emonstration de N. Nadirashvili des conjectures de Hadamard et Calabi-Yau.}
Bull. Sci. Math. \textbf{123} (1999),
no.7, 563--575.
\bibitem{courant} R. Courant, {\em Dirichlet's principle, conformal mapping and
minimal surfaces}, New York, Interscience, 1950.

%
\bibitem{dierkeseco} U. Dierkes, S. Hildebrandt, A. K{\"u}ster, O Wohlrab, {\em
Minimal surfaces. {I}.} Grundlehren der Mathematischen Wissenschaften
[Fundamental Principles of Mathematical Sciences] \textbf{295} (Boundary value
problems), Springer-Verlag, Berlin, 1992.
\bibitem{dierkeseco2} U. Dierkes, S. Hildebrandt, A. K{\"u}ster, O. Wohlrab,
{\em Minimal surfaces. {II}.} Grundlehren der Mathematischen Wissenschaften
[Fundamental  Principles of Mathematical Sciences] \textbf{296} (Boundary
regularity), Springer-Verlag, Berlin, 1992.

%
\bibitem{douglas} J. Douglas, {\em Solution of the problem of Plateau.} Trans.
Amer. Math. Soc. \textbf{33} (1931), 263--321

\bibitem{ferrer-martin-meeks} L. Ferrer, F. Martin, W. Meeks III, {\em Existence
of proper minimal surfaces of arbitrary topological type.} Preprint (2009).

\bibitem{Gr} A. Grigor'yan, {\em Analytic and geometric background of recurrence
and non-explosion of the Brownian motion on Riemannian manifolds.} Bull. Amer.
Math. Soc. \textbf{36} (1999), 135--249.

\bibitem{jorgetomi} L.P. Jorge and F. Tomi, {\em The barrier principle for minimal submanifolds of arbitrary codimension.} Ann. Global Anal. Geom. \textbf{24} (2003), no.3, 261--267.
\bibitem{jorge-xavier-annals} L.P. Jorge and F. Xavier, {\em A complete minimal
surface in $\mathbb{R}^{3}$ between two parallel planes.} Annals of
Math. (2) \textbf{112} (1980), 203--206.
\bibitem{jost} J. Jost, {\em Conformal mappings and the Plateau-Douglas problem
in Riemannian manifolds.} J. Reine Angew. Math. \textbf{359} (1985), 37--54.



\bibitem{pacomartin}F. Mart\'{\i}n, S. Morales, \textit{Complete proper minimal
surfaces in convex bodies of $\mathbb{R}^{3}$.} Duke Math. J. \textbf{128}
(2005), 559--593.

\bibitem{pacomartin2}F. Mart\'{\i}n,  S. Morales, \textit{Complete proper
minimal surfaces in convex bodies of $\mathbb{R}^{3}$. II. The behavior of the
limit set.} Comment. Math. Helv. \textbf{81} (2006), 699--725.

\bibitem{martinmorales} F. Mart{\'{\i}}n, S. Morales, \emph{Construction of a
complete bounded minimal annulus in $\Bbb
              R^3$.} Proceedings of the Meeting of Andalusian Mathematicians,
              Vol. II (Spanish) (Sevilla, 2000), Colecc. Abierta \textbf{52}
(2001), Univ. Sevilla Secr. Publ., Seville, 649--653.

\bibitem{martin-nadirashvili} F. Mart\'in, N. Nadirashvili, {\em A Jordan curve
spanned by a complete minimal surface.}  Arch. Ration. Mech. Anal. \textbf{184}
(2007), no. 2, 285--301.

\bibitem{mattila} P. Mattila, \emph{Geometry of sets and measures in Euclidean
spaces. Fractals and rectifiability}. Cambridge Studies in Advanced Mathematics
\textbf{44}, Cambridge University Press, Cambridge, 1995.

\bibitem{Nadirashvili} N. Nadirashvili,  {\em Hadamard's and
Calabi-Yau's conjectures on negatively curved and minimal
surfaces.} Invent. Math. \textbf{126} (1996), 457--465.

\bibitem{osgood} W.F. Osgood, {\em A Jordan curve of positive area.} Trans. Amer. Math. Soc. \textbf{4} (1903), no.1, 107--112.
\bibitem{persson} A. Persson, \emph{Bounds for the discrete part of the spectrum
of a semibounded Schr\"odinger operator.} Math. Scand. \textbf{8} (1960),
143--153.

\bibitem{prs} S. Pigola, M. Rigoli, A.G. Setti,
\emph{Vanishing and finiteness results in Geometric Analisis. A generalization
of the B\"ochner technique.} Progress in Math. \textbf{266}, Birk\"auser, 2008.

\bibitem{reedsimon3} M. Reed, B. Simon, \emph{Methods of Modern Mathematical
Physics. IV. Analysis of Operators.} Academic Press, New
York-London, 1978.

\bibitem{tokuomaru} M. Tokuomaru, \emph{Complete minimal cylinders properly
immersed in the unit ball.} Kyushu J. math. \textbf{61} (2007), no.2, 373--394.
\bibitem{tomitromba} F. Tomi, A.J. Tromba, {\em Existence theorems for minimal
surfaces of nonzero genus  spanning a contour.} Mem. Amer. Math. Soc.
\textbf{71} (1988), no.382, pp. iv+83.

\bibitem{werner} H. Werner, \emph{The existence of surfaces of constant mean
curvature with arbitrary Jordan curves as assigned boundary.} Proc. Amer. Math.
Soc. \textbf{11} (1960), 63--70.

\bibitem{Yau3} S.T. Yau, {\em Review of Geometry and Analysis.} Kodaira's issue.  Asian J. Math.
\textbf{4} (2000), 235--278.

\bibitem{yau4} S.T. Yau, {\em Review of Geometry and Analysis.} Mathematics:
frontier and perspectives. Amer. Math. Soc. Providence. RI. (2000), 353--401.

\end{thebibliography}
\end{document}